\documentclass{amsart}

	\usepackage{aliascnt}
	\usepackage[colorlinks=true, linkcolor=black, citecolor=black, urlcolor=blue]{hyperref}
	\usepackage{amssymb}	
	\usepackage{amsmath}
	\usepackage{amsthm}
	\usepackage{layout}

	\newaliascnt{lemma}{thm}
	  
	\aliascntresetthe{lemma}

	\newaliascnt{prop}{thm}
	\newtheorem{prop}[prop]{Proposition} 
	\aliascntresetthe{prop}

	\newaliascnt{cor}{thm}
	
	\aliascntresetthe{cor}

	\theoremstyle{remark}

	\newaliascnt{rem}{thm}
	\newtheorem{rem}[rem]{Remark}
	\aliascntresetthe{rem}

	\theoremstyle{definition}

	\newaliascnt{exm}{thm}
	
	\aliascntresetthe{exm}

	\newaliascnt{notn}{thm}
	
	\aliascntresetthe{notn}

	\newaliascnt{defn}{thm}
	
	\aliascntresetthe{defn}

	\newcommand{\K}{\mathbb{K}}
	\newcommand{\rk}{\operatorname{rk}}

	\newcommand{\Sy}{\operatorname{Sym}}

	\newcommand{\Ps}{\mathbb{P}}
	\newcommand{\Span}[1]{\left\langle\,#1\,\right\rangle}

\begin{document}
\title[Structures $\ldots$ Decompositions of Ternary Quartics]{Structures of the Length Seven Power Sum Decompositions of Ternary Quartics}
\author[Alessandro De Paris]{Alessandro De Paris}
\address{Dipartimento di Scienze Umane, Link Campus University, Roma (Italy)}
\address{ORCID: \href{https://orcid.org/0000-0002-4619-8249}{0000-0002-4619-8249}}
\email{a.deparis@unilink.it}

\begin{abstract}
Motivated by the search for a deeper understanding of tensor rank, in view of its computational complexity applications, we investigate a possible path to determine the maximum symmetric rank in given degree and dimension. We work in terms of Waring rank of forms, and aiming to set up a firm basis for an induction procedure we examine some technical tools to organize length seven Waring decompositions of ternary quartics.
\end{abstract}

\maketitle

\section{Introduction}

A problem of interest in Algebraic Complexity Theory is to find the maximum symmetric rank for symmetric tensors of given type. It can be equivalently cast in terms of Waring rank, and this way gives rise to a polynomial version of the Waring problem in Number Theory, which is extensively discussed in \cite{G} in connection with the Alexander-Hirschowitz theorem in Algebraic Geometry. A clear idea on the general context we are moving in can be obtained by giving a look on the books \cite{L1}, \cite{L2} and the lecture notes \cite{OR}. A small sample of recent results close to this subject is given by \cite{AC}, \cite{BT2}, \cite{CGLV}.

The gap between the best lower and upper bounds that are known to date for Waring ranks in the space $S_d$ of \emph{all} ternary forms of degree $d$ is quite narrow (see \cite[Sect.~6]{D3}). A possible answer to the problem is given at the end of \cite[Introduction]{D3}, and another intriguing possibility is
\begin{equation}\label{2guess}
\left\lfloor\frac{\dim S_d}2\right\rfloor
\end{equation}
(which vaguely resembles the Blekhermann-Teitler upper bound established in \cite{BT}). Both guesses agree with the known values for $d$ up to $5$, and with the known asymptotic estimate $d^2/4$. The first $d$ for which we get two different answers is $6$: the first prediction gives $13$, and the other $14$. In spite of the promising outcomes of the recently set up techniques for finding upper bounds, even the simplest unknown case $d=6$ has resisted our investigation. Much less effort has been devoted to check the second possibility, but at the moment we do not see how \eqref{2guess} could give the right answer to the problem.

In this situation, an attempt that we believe is worthy of being pursued is a deeper study of the structure of some kind of power sum decompositions that are suggested by the guiding picture in \cite{D3}. For sextics, the projective geometric description of such decompositions is obtained by considering five distinct lines, two of them containing each three points, the other three lines containing two points, and an extra point outside the lines. A dimension count shows that $d=6$ is the first level that requires a nongeneric choice of these lines, and in order to refine the inductive method set up in \cite{D3}, one needs a better control of the decompositions at the level $d=4$. For quartics, note that seven points in linear generic position (in a plane) can always be arranged in three pairs on three distinct lines and an extra point outside, and this is also true for most of the linear special positions.

In this work we describe a systematic procedure to detect such decompositions of quartics, which has some chance of being extended to higher degrees, and constitute a framework in which power sum decompositions of length seven can be organized. But perhaps the most interesting fact is that there are special quartics that behaves oddly with respect to what is expected in that framework.

\section{Notation}

We assume the notation of \cite[Sect.~2]{D3}. In particular, throughout the paper we keep the notation $\K$ for an algebraically closed field of characteristic zero and $S^\bullet=\Sy^\bullet S^1$, $S_\bullet=\Sy^\bullet S_1$ for two symmetric $\K$-algebras between which an \emph{apolarity pairing} is given (for details see \cite[Introduction]{D1}). Basically, $S^\bullet$ and $S_\bullet$ can be regarded as rings of polynomials in a finite and the same number of indeterminates, acting on each other by constant coefficients partial differentiation. For each $x\in S^\bullet$ and $f\in S_\bullet$ we shall denote by $\partial_xf$ the apolarity action of $x$ on $f$. The evaluation of a homogeneous form $x\in S^d$ on $v\in S_1$ is given by
\[
x(v):=\frac{\partial_xv^d}{d!}\,.
\]
The (Waring) \emph{rank} of $f\in S_d$, $d>0$, denoted by $\rk f$, is the least of the numbers $r$ such that $f$ can be written as a sum of $r$ $d$-th powers of forms in $S_1$. The span of $v_1,\ldots,v_n$ in some vector space $V$ will be denoted by $\Span{v_1,\ldots,v_n}$, and we assume as a formal definition of the projective space $\Ps V$, the set of all one-dimensional subspaces $\Span{v}\subseteq V$, $v\ne 0$. The sign $\perp$ will refer to orthogonality with respect to the apolarity pairing $S^d\times S_d\to\K$, when some degree $d$ is fixed (sometimes implicitly).

We also explicitly recall from \cite[Def.~2.2]{D3} that if a decomposition
\begin{equation}\label{D}
f=\lambda_1{v_1}^d+\cdots+\lambda_r{v_r}^d\;,\quad\lambda_1,\ldots,\lambda_r\in\K,\;v_1,\ldots,v_r\in S_1\;,
\end{equation}
is given and $x\in S^\delta$ vanishes on no one of $v_1,\ldots ,v_r$, then \emph{the} $x$-antiderivative of $f$ relative to \eqref{D} is
\[
F:=\frac{d!\lambda_1}{(d+\delta)!x\left(v_1\right)}{v_1}^{d+\delta}+\cdots+\frac{d!\lambda_r}{(d+\delta)!x\left(v_r\right)}{v_r}^{d+\delta}\;;
\]
when the powers ${v_1}^d,\ldots, {v_r}^d$ are linearly independent, we also say that the antiderivative is relative to $v_1,\ldots, v_r$. The name is due of course to $\partial_xF=f$.

\section{A Framework for Length Seven Power Sum\break Decompositions of Ternary Quartics}

\begin{rem}\label{uno}
Assume that $\dim S_1=2$. Let $\Span q\in\Ps S_2$, $\Span x\in\Ps S^1$, and let \[L:=S_3\cap\partial_{x}^{-1}\left(\Span{q}\right)\;.\] According to \cite[Lemma~2.6]{D3}, if $q$ is not a square, then there exists an isomorphism of projective spaces \[\omega:\Ps\Span q^\perp\overset\sim\to\Ps L\] such that $\partial_hF_h=0$ when $\Span{F_h}=\omega(\Span h)$. The stronger property
\begin{equation}\label{Prope}
\Span{F}=\omega(\Span h)\;\iff\;\partial_hF=0
\end{equation}
(with $\Span h\in\Ps\Span q^\perp$, $\Span F\in\Ps\Span L$) can easily be checked.
\end{rem}

\begin{rem}\label{due}
Assume that $\dim S_1=3$. Let $c\in S_3$, $x^0,x^1\in S^1$ be linearly independent, and suppose that \[\partial_{x^0}c=:q_1\ne 0\;,\qquad\partial_{x^1}c=:q_0\ne 0\;,\qquad\partial_{x^0x^1}c=0\;.\] For $i\in\{0,1\}$ let us consider the rings $S_{i,\bullet}:=\ker\partial_{x^i}$, $S_i^\bullet:=S^\bullet/\left(x^i\right)$, with the apolarity pairing induced by that of $S_\bullet$ and $S^\bullet$, and set
\[
L_i:=S_{i,3}\cap\partial_{x^{1-i}}^{-1}\left(\Span{q_i}\right)\;.\]
Given $F\in L_0$, we have $\partial_{x^1}F=\lambda(F)q_0$ for some scalar $\lambda(F)$, so that a linear form $\lambda$ on $L_0$ is defined. This allows us to define a vector space isomorphism
\[
\psi:L_0\to L_1\;,\qquad F\;\mapsto\;\lambda(F)c-F
\]
(cf.\ the proof of \cite[Lemma~2.7]{D3}), which gives an isomorphism of projective spaces \[\Ps\psi:\Ps{L_0}\to\Ps{L_1}\;.\] We have $c\in\Span{F,\psi(F)}$ for all $\Span{F}$ in $\Ps{L_0}$ except the point $\Span{v^3}\in\Ps{L_0}\cap\Ps{L_1}$, with $\Span{v}:=\Span{x^0,x^1}^\perp$, which is a fixed point of $\Ps\psi$.\\
For $i\in\{0,1\}$, if $q_i$ is not a square, according to \autoref{uno} with $S_{i,\bullet}$ and $S^i_{\bullet}$ in place of $S_\bullet$ and $S^\bullet$, respectively, and the coset $x^{1-i}+\left(x^i\right)$ in place of $x$, we also have an isomorphism of projective spaces
\[
\omega_i:\Span{q_i}^\perp\to\Ps L_i\;,
\]
where the orthogonal complement is understood inside $S_i^2$, not in $S^2$ (otherwise, one may write $\left(\,\Span{q_i}^\perp+\left(x^i\right)\,\right)/\left(x^i\right)$).
We end up with an isomorphism of projective spaces
\[
\theta:=\omega_1^{-1}\,\circ\,\Ps\psi\,\circ\,\omega_0:\Ps\Span{q_0}^\perp\to\Ps\Span{q_1}^\perp
\]
whenever $q_0,q_1$ are not squares. If $\partial_{h^0}\left(F_0\right)$ for some $\Span{h^0}\in\Ps\Span{q_0}^\perp$, $\Span{F_0}\in\Ps L_0$ and $\theta\left(\Span{h^0}\right)=\Span{h^1}$, we have
\begin{equation}\label{Cond}
    \Span{F_1}=\Span{\psi\left(F_0\right)}\qquad\iff\qquad\partial_{h^1}\left(F_1\right)=0\;,\quad\Span{F_1}\in\Ps L_1
\end{equation}
(because of \eqref{Prope} and the definition of $\theta$).
\end{rem}
 
\begin{prop}\label{main}
Suppose that $\dim S_1=3$, $f\in S_4$ and that $x^0, x^1,x^2\in S^1$ are linearly independent forms such that $\partial_{x^0x^1x^2}f=0$.\\
Let $\sigma$ be the cyclic permutation $0\mapsto 1,\,1\mapsto 2,\,2\mapsto 0$, for each $i\in\{0,1,2\}$ set $i':=\sigma(i)$, $i'':=\sigma\left(i'\right)$, \[q_i:=\partial_{x^{i'}x^{i''}}f\;,\qquad \Span{v_i}:=\Span{x^{i'},x^{i''}}^\perp\;,\] and with reference to the dually paired rings $S_{i,\bullet}:=\ker\partial_{x^i}$, $S_i^\bullet:=S^\bullet/\left(x^i\right)$, let \[\theta_i:\Ps\Span{q_i}^\perp\to\Ps\Span{q_{i'}}^\perp\] be defined as in \autoref{due} with $\partial_{x^{i''}}f$ in place of $c$ and $x^i,x^{i'}$ in place of $x^0,x^1$, respectively. Finally, let us suppose that
\begin{itemize}
\item $q_0$, $q_1$, $q_2$ are not squares;
\item $\Span{h^0}$ is a fixed point of $\theta_2\circ\theta_1\circ\theta_0$ (in other words, $h^0$ is an eigenvector of an underlying vector space automorphism);
\item the roots of $\Span{h^0}$, $\Span{h^1}:=\theta_1\left(\Span{h^0}\right)$ and $\Span{h^2}:=\theta_2\left(\Span{h^1}\right)$ in the lines $\Ps S_{0,1}$, $\Ps S_{1,1}$ and $\Ps S_{2,1}$, respectively, are distinct and different from $\Span{{v_0}^2}$, $\Span{{v_1}^2}$ and~$\Span{{v_2}^2}$.
\end{itemize}
Then $f$ admits a decomposition as a sum of exactly six fourth powers of linear forms, each spanning the above mentioned roots. 
\end{prop}
\begin{proof}
For each $i\in\{0,1,2\}$ let $P_i, Q_i$ be the roots of $\Span{h_i}$ in $\Ps S_{i,1}$. Since $h_i\in\Span{q_i}^\perp$, $P_i\ne Q_i$ and $q_i$ is not a square, the apolarity lemma gives decompositions
\begin{equation}\label{Dec}
    q_i={u_i}^2+{w_i}^2\;,\text{ with } \Span{u_i}=P_i,\Span{w_i}=Q_i\qquad\forall i\in\{0,1,2\}\;.
\end{equation}
For each $i\in\{0,1,2\}$, let $f_i$ be the $x^{i'}x^{i''}$-antiderivative of $q_i$ relative to $u_i,w_i$ and note that $\partial_{x^i}f_i=0$. We have
\begin{equation}\label{start}
\partial_{x^i}\left(f_0+f_1+f_2\right)=F_{ii'}+F_{ii''}\;,
\end{equation}
with
\[
F_{ii'}:=\partial_{x^i}f_{i'}\;,\qquad F_{ii''}:=\partial_{x^i}f_{ii''}\;,
\]
so that $F_{ii'}$ is the $x^{i''}$-antiderivative of $q_{i'}$ and $F_{ii''}$ is the $x^{i'}$-antiderivative of $q_{i''}$ (relative to \eqref{Dec}), and $\partial_{x^{i'}}F_{ii'}=0$, $\partial_{x^{i''}}F_{ii''}=0$.
Let
\[
L_{ii'}:=S_{i',3}\cap\partial_{x^{i''}}^{-1}\left(q_{i'}\right)\;,\qquad L_{ii''}:=S_{i'',3}\cap\partial_{x^{i'}}^{-1}\left(q_{i''}\right)\;,
\]
so that
\[
F_{ii'}\in L_{ii'}\;,\qquad F_{ii''}\in L_{ii''}\;,
\]
and let
\[
\psi_i:L_{ii'}\to L_{ii''}
\]
    be defined as in \autoref{due} with $\partial_{x^i}f$ in place of $c$ and $x^{i'},x^{i''}$ in place of $x^0,x^1$, respectively. Since $\Span{h^0}$ is a fixed point of $\theta_2\circ\theta_1\circ\theta_0$, we have that $\theta_2\left(\Span{h^2}\right)=\Span{h^0}$, so that $\Span{\theta_{i'}\left(h_{i'}\right)}=\Span{h_{i''}}$ for whatever choice of $i\in\{0,1,2\}$. On the other hand we have $\partial_{h^{i'}}\left(F_{ii'}\right)=0$, because of the decomposition~\eqref{Dec} (with $i'$ in place of $i$), the fact that $\Span{u_{i'}},\Span{w_{i'}}$ are the roots of $h^{i'}$ and the definition of the antiderivative relative to $u_{i'}$, $w_{i'}$; similarly, $\partial_{h^{i''}}\left(F_{ii''}\right)=0$. According to \eqref{Cond}, this shows that $\Span{F_{i''}}=\Span{\psi_{i'}\left(F_{i'}\right)}$, and consequently that $\partial_{x^i}f\in\Span{F_{ii'},F_{ii''}}$ (by \autoref{due} and the fact that $\Span{F_{ii'}}\ne\Span{{v_i}^3}$ because of its decomposition in terms of $u_{i'},w_{i'}$).

Since $F_{ii'}+F_{ii''}\in\Span{F_{ii'},F_{ii''}}$ as well, we have
\[
\partial_{x^i}f-\left(F_{ii'}+F_{ii''}\right)\in\Span{F_{ii'},F_{ii''}}\,.
\]
Moreover,
\[
\partial_{x^{i'}}\partial_{x^i}f=q_{i''}=\partial_{x^{i'}}F_{ii''}=\partial_{x^{i'}}F_{ii''}+\partial_{x^{i'}}F_{ii'}\;,
\]
\[
\partial_{x^{i''}}\partial_{x^i}f=q_{i'}=\partial_{x^{i''}}F_{ii'}=\partial_{x^{i''}}F_{ii'}+\partial_{x^{i''}}F_{ii''}\;,
\]
hence $\partial_{x^i}f$ and $F_{ii'}+F_{ii''}$ differ by a multiple of ${v_i}^3$, that is
\[ 
\partial_{x^i}f-\left(F_{ii'}+F_{ii''}\right)\in\Span{{v_i}^3}\;.
\]
We have shown that
\[
\partial_{x^i}f-\left(F_{ii'}+F_{ii''}\right)\in\Span{F_{ii'},F_{ii''}}\cap\Span{{v_i}^3}\;,
\]
but that intersection is the zero space because ${v_i}^3\not\in\Span{F_{ii'},F_{ii''}}$ (since $\partial_{x^{i'}}v_i=0$, $\partial_{x^{i'}}F_{ii'}=0$, $\partial_{x^{i'}}F_{ii''}=q_{i''}\ne 0$, and $\partial_{x^{i''}}v_i=0$, $\partial_{x^{i''}}F_{ii'}=q_{i'}\ne 0$, $\partial_{x^{i''}}F_{ii''}=0$).

Recalling \eqref{start} we deduce that
\[
\partial_{x^i}f=\partial_{x^i}\left(f_0+f_1+f_2\right)
\]
for all $i\in\{0,1,2\}$, and since $x^0, x^1,x^2\in S^1$ are linearly independent, this means that
\[
f=f_0+f_1+f_2\;.
\]
To conclude the proof is just to point out that $f_i$ is a sum of two fourth powers of multiples of $u_i$ and $w_i$, and $\Span{u_i}=P_i,\Span{w_i}=Q_i$ are the roots of $h^i$.
\end{proof}

\begin{rem}\label{procedura}
The above proposition suggests a procedure to decompose a given $f\in S_4$, with $\dim S_1=3$ as a sum of exactly seven fourth powers of linear forms:
\begin{itemize}
	\item consider linearly independent $x^0, x^1,x^2\in S^1$ and set $v:=\partial_{x^0x^1x^2}f$;
	\item if $\partial_{x^i}v\ne 0$ for all $i\in\{0,1,2\}$, the $x^0x^1x^2$-antiderivative $V$ of $v$ (relative to itself) is defined, and is a fourth power of a multiple of $v$;
	\item set $g:=f-V$, so that $g\in S_4$ and $\partial_{x^0x^1x^2}g=0$, as required to exploit \autoref{main} for $g$ (in place of $f$);
	\item if the hypotheses listed at the end of the proposition are satisfied, we obtain a length six power sum decomposition of $g$, hence a length seven decomposition of $f$.
\end{itemize}
\end{rem}

The conditions needed in the above procedure are quite mild: for a generic choice of $x^0$, $x^1$, $x^2$, $v$ works well, generic quadratic forms are not squares, and if the roots mentioned in the last condition turns out to be sufficiently generic, they are distinct and avoid the forbidden three points. One might hope to find in this way plenty of length seven decomposition of whatever given quartic. This would suggest a procedure in higher degree, and also constitutes a basis for an induction process. But, as a matter of facts, there are very special quartics for which the procedure fails. To understand those special cases may be helpful to determine if the procedure can be refined, or has to be abandoned.

\section{Special cases}

In this section we keep the notation of \autoref{procedura}.

A rather trivial situation is when $f$ is a fourth power. In this case, $x^0$, $x^1$, $x^2$, for which $v$ works well can easily be found, unless $f=0$, but $g$ always vanishes (hence the forms $q_0$, $q_1$ and $q_2$, obtained by trying to exploit \autoref{main} for $g$, are trivially squares).

When $f$ has rank two, that is, it is the sum of two fourth power of two linearly independent linear forms, the forms $q_0$, $q_1$ and $q_2$, obtained by trying to exploit \autoref{main}, are always squares. The proof is not immediate, but not difficult and we skip it. We also skip the analysis of rank three forms, and directly pass to the perhaps oddest case in our knowledge. Indeed \cite[Examples~3.3 and~3.4]{BD} give many $f\in S_4$ for which every length seven power sum decomposition must involve some $v\in S_1$ for which the point $\Span{v}\in\Ps S_1$ must lie on some special lines in $\Ps S_1$ (depending on $f$). Thus our procedure for length seven decompositions either produces points on the special lines for every choice of $x^0$, $x^1$, $x^2$ for which the procedure works, or fails for all choices of $x^0$, $x^1$, $x^2$: an odd behavior in both cases. We performed computational experiments, and an unexpected outcome has been that the composition $\theta_2\circ\theta_1\circ\theta_0$ always had only one fixed point (instead of two, as a generic automorphism of a projective lines). Moreover, in one case we checked all the roots, and found that there was a root lying on the predicted special line on each of the three lines $x^0=0$, $x^1=0$, $x^2=0$.

Another potentially critical situation is suggested by \cite[Proposition~3.2]{D0}: a length seven decomposition of a quartic made by a double line and a conic not tangent to the line, must have the linear forms in a special position. The procedure did not fail, and produced a decomposition in the expected special position. Remarkably, again the composition $\theta_2\circ\theta_1\circ\theta_0$ had only one fixed point. We also checked the procedure on a form taken at random, and this time the composition $\theta_2\circ\theta_1\circ\theta_0$ had two fixed points.

\end{document}